\documentclass{amsart}
\usepackage{amssymb,esint}
\usepackage{amscd}
\usepackage{xspace}
\usepackage{fancyhdr}
\usepackage{xcolor}
\usepackage{hyperref,amsmath}
\usepackage{cite}
\newtheorem{theorem}{Theorem}[section]
\newtheorem{lemma}[theorem]{Lemma}
\newtheorem{corollary}[theorem]{Corollary}
\theoremstyle{definition}
\newtheorem{definition}[theorem]{Definition}

\theoremstyle{proposition}
\newtheorem{proposition}[theorem]{Proposition}
\theoremstyle{remark}
\newtheorem{remark}[theorem]{Remark}

\numberwithin{equation}{section}

\begin{document}

\title{The Muskat problem with $C^1$ data}

\author{Ke Chen}
\address{Fudan University, 220 Handan Road, Shanghai, 200433, China.}
\email{kchen18@fudan.edu.cn}

\author{Quoc-Hung Nguyen}
\address{ShanghaiTech University,
	393 Middle Huaxia Road,
	Shanghai, 201210, China.}
\curraddr{Academy of Mathematics and Systems Science, Chinese Academy of Sciences, Beijing, 100190, China. }
\email{qhnguyen@amss.ac.cn}
\thanks{Quoc-Hung Nguyen is supported by the ShanghaiTech University startup fund and the National Natural Science Foundation of China (12050410257). This work is finished during Ke Chen and Yiran Xu's visit to ShanghaiTech.}

\author{Yiran Xu}
\address{Fudan University, 220 Handan Road, Shanghai, 200433, China.}
\email{yrxu20@fudan.edu.cn}

\date{ xxxxx and, in revised form, xxxx.}



\begin{abstract}
 In this paper we prove that the Cauchy problem of the Muskat equation is wellposed locally in time for any initial data in $\dot C^1(\mathbb{R}^d)\cap L^2(\mathbb{R}^d)$.
\end{abstract}

\maketitle
	\section{Introduction}
The Muskat equation is an important model in the analysis of free surface flows, which 
describes the dynamics of two incompressible and immiscible fluids with different densities and viscosities separated by a porous media whose velocities obey 
Darcy's law (see \cite{darcy1856fontaines},\cite{Mus34}). 
Its main feature is  that it is a  fractional degenerate parabolic equation.
This feature is shared by several equations 
which have attracted a lot of attention in recent years, like the surface quasi-geostrophic equation, 
the Hele-Shaw equation and	the fractional porous media equation.

Let us introduce the Muskat problem. We consider the dynamics of a time-dependent curve $\Sigma(t)$ separating two domains $\Omega_1(t)$ and $\Omega_2(t)$. 
Under the supposition that $\Sigma(t)$ is the graph of some function, we introduce the following notations
\begin{align*}
	\Omega_1(t)&=\left\{ (x,y)\in \mathbb{R}^{d}\times \mathbb{R}\,;\, y>f(t,x)\right\},\\
	\Omega_2(t)&=\left\{ (x,y)\in \mathbb{R}^{d}\times \mathbb{R}\,;\, y<f(t,x)\right\},\\
	\Sigma(t)&=\left\{ (x,y)\in \mathbb{R}^{d}\times \mathbb{R}\,;\, y=f(t,x)\right\}.
\end{align*}
Assume 
that each domain~$\Omega_j$, $j=1,2$, 
is occupied by an incompressible fluid with constant density~$\rho_j$ and 
denote $\rho=\rho_1\mathbf{1}_{\Omega_1(t)}+\rho_2\mathbf{1}_{\Omega_2(t)}$. 
Then the motion is determined by the incompressible porous media equations, where the velocity 
field $v$ is given by Darcy's law:
\begin{equation*}
	\left\{
	\begin{aligned}
		&\partial_t\rho+\text{div}(\rho v)=0,\quad\text{	(transport equation)}\\
		& \text{div}(v)=0,\quad \quad \quad \ \quad  \text{(incompressible condition)}\\
		&v+\nabla (P+\rho gy)=0,\quad~\text{	(Darcy's law)}
	\end{aligned}
	\right.
\end{equation*}
where $g>0$ is the acceleration of gravity. 

Changes of unknowns, reducing the problem to an evolution 
equation for the free surface parametrization, 
have been known for quite a long time  (see~\cite{CaOrSi-SIAM90,EsSi-ADE97,Siegel2004}). 
These approaches were further developed by C\'ordoba and Gancedo \cite{Cordoba2007Contour}
who obtained a beautiful compact formulation of the Muskat equation. Indeed, 
they showed that the Muskat problem is equivalent to the following 
equation for the free surface elevation:
\[
\partial_t f(t,x)=\frac{\rho_2-\rho_1}{2^{d}\pi}\text{	P.V.}\int_{\mathbb{R}^d}\frac{\alpha\cdot\nabla_x\Delta_\alpha f(t,x)}{\left\langle \Delta_\alpha f(t,x)\right\rangle^{d+1}}\frac{d\alpha}{|\alpha|^d},
\]
where 
the integral is understood in the sense of principal values, $\Delta_\alpha f$ is the slope defined by $
\Delta_\alpha f(x)=\frac{f(x)-f(x-\alpha)}{|\alpha|},$
and $\langle a \rangle=(1+a^2)^\frac{1}{2}$.\\
At a linear level, the Muskat equation reads
\begin{equation*}
\begin{aligned}
	\partial_t f(t,x)=\frac{\rho_2-\rho_1}{2^{d}\pi}\text{	P.V.}\int_{\mathbb{R}^d}\alpha\cdot\nabla_x\Delta_\alpha f(t,x)\frac{d\alpha}{|\alpha|^d}=-\frac{\rho_2-\rho_1}{2}|D| f(t,x).
\end{aligned}	
\end{equation*}

The problem is said to be in stable regimes (heavier 
fluid below) if $\rho_2>\rho_1$ and unstable regimes (heavier 
fluid on the top) if $\rho_2<\rho_1$. In this paper, 
we consider the problem in stable regimes, i.e. $\rho_2>\rho_1$. 
In order to simplify the exposition we take $\frac{\rho_2-\rho_1}{2^{d}\pi}=1$. This leads to the equation
\begin{equation}\label{E1}
	\begin{aligned}
		\partial_t f(t,x)=\text{	P.V.}\int_{\mathbb{R}^d}\frac{\alpha\cdot\nabla_x\Delta_\alpha f(t,x)}{\left\langle \Delta_\alpha f(t,x)\right\rangle^{d+1}}\frac{d\alpha}{|\alpha|^d}.\\
	\end{aligned}
\end{equation}
Recall that the Muskat equation is invariant under the change of unknowns:
\[
f(t,x)\rightarrow f_\lambda(t,x):=\frac{1}{\lambda}f(\lambda t,\lambda x).
\]
By a direct calculation, one verifies that the spaces $\dot H^{1+\frac{d}{2}}(\mathbb{R}^d),\dot W^{1,\infty}(\mathbb{R}^d)$ are two critical spaces for the Cauchy problem of the Muskat equation \eqref{E1}.

The analysis of the Cauchy problem for the Muskat equation is now well developed, 
including global existence results under mild smallness assumptions and blow-up 
results for some large enough initial data. 
Local well-posedness results can be traced  back to the works of Yi~\cite{Yi2003}, 
Ambrose~\cite{Ambrose2004,Ambrose2007}, 
C\'ordoba and Gancedo~\cite{Cordoba2007Contour,Cordoba2009}, C\'ordoba, C\'ordoba and Gancedo~\cite{Cordoba2011}, 
Cheng, Granero-Belinch\'on, 
Shkoller~\cite{21}. 
Local well-posedness results  in 
the sub-critical spaces were obtained by 
Constantin, Gancedo, Shvydkoy and Vicol~\cite{25} 
for initial data in the Sobolev space 
$W^{2,p}(\mathbb{R})$ for some $p>1$, Ables-Matioc \cite{Ables-Matioc} for $W^{s,p}$ with $s>1+\frac{1}{p}$, and Nguyen-Pausader \cite{43}, Matioc~\cite{Matioc1,Matioc2}, Alazard-Lazar~\cite{Alazard2020Paralinearization}  for initial data 
in $H^s(\mathbb{R})$ with $s>3/2$. 
Since the Muskat equation is parabolic, the proof of the local well-posedness 
results also gives global well-posedness results under a smallness assumption, see Yi~\cite{Yi2003}. 
The first global well-posedness results 
under mild smallness assumptions, namely assuming 
that the Lipschitz semi-norm is smaller than $1$, was 
obtained by Constantin, C{\'o}rdoba, Gancedo, Rodr{\'\i}guez-Piazza 
and Strain~\cite{Constantin2010,Constantin2013,25} (see also \cite{Cameron2019}). They also proved the existence of global classical solution  when the initial data $f_0\in H^3(\mathbb{R})$ with the Wiener norm $\|f_0\|_{\mathcal{L}^{1,1}}=\||\xi|\hat f_0(\xi)\|_{L^1_\xi(\mathbb{R})}$ less than some explicit constant (see also \cite{Constantin2010} which improved the constant to $\frac{1}{3}$). Note that there exists finite time blow up solution for the general non-graph interface(see \cite{CCF+12,Castro2013}). However, 
it is possible to solve the Cauchy problem for initial 
data whose slope can be arbitrarily large. 
Deng, Lei and Lin in~\cite{33} obtained the first result in this direction, under the assumption that the 
initial data are monotone (see also Remark \ref{rem1}).  C\'ordoba and Lazar in ~\cite{Cordoba-Lazar-H3/2} proved global  well-posedness for the 2D Muskat equation in critical space $H^{1+\frac{d}{2}}\cap W^{1,\infty}$ with $\dot H^{1+\frac{d}{2}}$ norm small, then Gancedo and Lazar \cite{Gancedo} extended the result to 3D. Recently, Alazard and the second author \cite{		ThomasHfirst,Alazard2020,Alazard2020endpoint,TH4} proved well-posedness for the Muskat equation with unbounded slopes. In particular, in \cite{Alazard2020endpoint} they  obtained global well-posedness  with $\dot H^{\frac{3}{2}}$  norm of initial data small and  local well posedness for large data in  $\dot H^{\frac{3}{2}}$, where they used a null-type structure to compensate the degeneracy of the parabolic behavior. In 3D, the existence of global classical solutions was established when the initial data satisfies $\|f_0\|_{\mathcal{L}^{1,1}}\leq \frac{1}{5}$. We also note that this result has been extended in \cite{viscojump} to a more general scenario where the viscosity of the two fluids can be different. More recently,  the global existence was proved in \cite{Cameron2020}, under the assumption $\|\nabla f_0\|_{L^\infty}\leq\frac{1}{\sqrt{5}}$. The existence of self-similar solutions with small initial data can be found in \cite{Self}. We also refer interested readers to \cite{TH5}, \cite{peskin} for other non-local parabolic equations.


At the moment, there is no result about local well-posedness for large data in $\dot W^{1,\infty}(\mathbb{R}^d)$ or Wiener space $\mathcal{L}^{1,1}(\mathbb{R}^d)$.  So it is quite interesting to study whether the problem is wellposed in the critical space $\dot C^1(\mathbb{R}^d)$ without any smallness assumptions on the data. 

We state the main result of our paper as follows
\begin{theorem} \label{mainthm1}For any $c_0>0$, there exists $\sigma=\sigma(c_0)\in (0,1]$ such that for any initial data $f_0\in L^2(\mathbb{R}^d)\cap \dot W^{1,\infty}(\mathbb{R}^d)$ with $\|\nabla f_0\|_{L^\infty}\leq c_0$, and $f_0$ can be decomposed as 
	\begin{equation}\label{initialdecompose}
		f_0=f_{0,1}+f_{0,2},	\quad\text{with}\quad \quad\|\nabla f_{0,1}\|_{L^\infty}\leq \sigma\quad and \quad f_{0,2}\in H^{10d},
	\end{equation}
	there exists a solution $ f$ of the Cauchy problem \eqref{E1} in $[0,T]$ for  $T=T(\|f_{0,2}\|_{H^{10d}},c_0)>0$, satisfying 
	\begin{equation*}
		\sup_{t\in[0,T]}||f(t)||_{L^2}\leq ||f_0||_{L^2},
			\end{equation*}
			\begin{equation*}
		\sup_{t\in [0,T]}||\nabla f(t)||_{L^\infty}	+\sup_{t\in [0,T]} t||f(t)||_{\dot C^{\frac{3}{2}}}\leq C(\|f_{0,2}\|_{H^{10d}},c_0).
	\end{equation*}
	Moreover, the solution $f$ can be decomposed as 
	\begin{align*}
		f=F_1+F_2, \quad\quad\text{with} \quad\| F_1\|_{L^\infty([0,T];\dot W^{1,\infty})}\leq 10 d\sigma \quad\text{and}\quad F_2\in L^\infty([0,T], H^{10d}).
	\end{align*}
\end{theorem}
\begin{remark}
	We remark that using the standard regularity theory, we can prove that the solution $ f$ of \eqref{E1}  in the class $L^\infty_{\text{loc}} ((0,T],C^{1+})$ will belong to $L^\infty_{\text{loc}} ((0,T],C^{\infty})$. Therefore, the solution in Theorem \ref{mainthm1} satisfies $f(t)\in C^\infty(\mathbb{R}^d)$ for any $0<t\leq T$.
\end{remark}
\begin{remark}
	It is possible that $f_{0,2}\in H^s(\mathbb{R}^d)$ for $s>1+\frac{d}{2}$ is enough for our results. But in this paper we will not discuss this in detail.
\end{remark}
Note that the decomposition \eqref{initialdecompose} holds for any  $f_0\in L^2(\mathbb{R}^d)\cap \dot C^1(\mathbb{R}^d)$, hence we have 
\begin{corollary}
	The statement in Theorem \ref{mainthm1} holds for initial data $f_0\in L^2(\mathbb{R}^d)\cap \dot C^1(\mathbb{R}^d)$.
\end{corollary}

The following proposition implies that the solution in Theorem \ref{mainthm1} is unique. 
\begin{proposition}\label{propunique}
	For any $c>0$, there exists $\sigma=\sigma(c)>0$ such that, if $f, \bar f\in L^\infty ([0,T], \dot W^{1,\infty})$ are solutions of the Muskat equation \eqref{E1} with $\|\nabla  f\|_{L^\infty_{t,x}}\leq c$ and the solution $\bar f$ can be decomposed as 
	\begin{align*}
		\bar f=\bar f_1+\bar f_2, \quad\quad \text{with} \quad\quad \|\nabla\bar f_1\|_{L^\infty([0,T;\dot W^{1,\infty}])}\leq \sigma \quad\  \text{and}\quad\  \bar f_2\in L^\infty([0,T], H^{10d}).
	\end{align*}
	Then we have
	\begin{equation*}
		\sup_{t\in [0,T]}||f(t)-\bar{f}(t)||_{L^\infty}\leq C ||(f-\bar{f})|_{t=0}||_{L^\infty},
	\end{equation*}
	where the constant $C$ depends on $c$ and $||\nabla\bar{f}_2||_{L^\infty_{t,x}}$.
\end{proposition}

We note that unless specified, all the integrals in this paper are understood as principal value integrals over $\mathbb{R}^d$.
We reformulate the equation as
\begin{equation*}
	\partial_t f(t,x)=\int\frac{\alpha\cdot\nabla f(t,x)-\delta_\alpha f(t,x)}{\left\langle \Delta_\alpha f(t,x)\right\rangle^{d+1}}\frac{d\alpha}{|\alpha|^{d+1}},
\end{equation*}
where $\delta_\alpha f(t,x)=f(t,x)-f(t,x-\alpha)$.
%
\vspace{0.2cm}\\

We organize the paper as follows. In the rest of this section, we will introduce the regularized system and notations which will be used throughout the paper.  In Section \ref{secprioriest} we will establish $L^2$ and Lipschitz estimates for the regularized system. Section \ref{ImproveREG} is devoted to improve the regularity, which helps to control the remainder terms in Section \ref{secprioriest}. We complete the proof of the main theorem in Section \ref{sectcomplete} and prove the  uniqueness result Proposition \ref{propunique} in Section \ref{sectionunique}. Finally, we prove the Proposition \ref{thm1} in the appendix.
\subsection{Regularization}
In order to rigorously justify the computations, we want to deal with smooth solutions. To achieve this, we introduce an approximate Muskat equation for which the Cauchy problem is easily studied, and whose solutions are expected to converge to solutions of the original Muskat equation. 

In this section, we follow the strategy introduced in \cite{Alazard2020}, which regularizes the Muskat equation depending on some parameters $\mu_1, \mu_2\in (0,1]$:
\begin{itemize}
	\item  Add a parabolic term of order 2 with a small viscosity of size $\mu_1$.
	\item Introduce a cut-off function in the singular integral to remove wave-length shorter than some parameter $\mu_2$.
\end{itemize}
More precisely, we introduce the following Cauchy problem
\begin{equation}\label{appMuskat}
	\left\{\begin{aligned}
		&\partial_t f(t,x)-\mu_1\Delta f(t,x)=\int\frac{\alpha\cdot\nabla f(t,x)-\delta_\alpha f(t,x)}{\left\langle \Delta_\alpha f(t,x)\right\rangle^{d+1}}(1-\chi\left({\alpha}/{\mu_2}\right))\frac{d\alpha}{|\alpha|^{d+1}},\\
		&f(0,x)=f_0(x),
	\end{aligned}\right.
\end{equation}
where $\chi: \mathbb{R}^d\rightarrow [0,1]$ is a smooth radial function such that
\begin{align*}
	\chi(y)=1 \quad\text{if}\quad 0\leq |y|\leq 1,\quad \chi(y)=0 \quad\text{if}\quad |y|\geq 2.
\end{align*}
Denote $\chi'(y)=\partial_r\chi(y)$. 
It is easy to verify that $x\cdot\nabla\chi(x)=\chi'(x)|x|$. For simplicity we denote $\chi_{\mu_2}(y)=\chi\left(\frac{y}{\mu_2}\right)$. We also assume that $-2\leq\chi'(x)\leq 0$. 
We have the following basic results\cite[Proposition 2.1]{Alazard2020}, \cite[Proposition 2.3]{Constantin2010}: 
\begin{proposition}
	For any $\mu_1, \mu_2\in (0,1]$ and any initial data $f_0\in L^2(\mathbb{R}^d) \cap \dot C^1(\mathbb{R}^d)$, the Cauchy problem \eqref{appMuskat} has a unique global solution  $f\in C^1([0,+\infty);H^\infty(\mathbb{R}^d))$.
\end{proposition}
\begin{proposition} \label{Cauchyofapp}
	Let $f_0\in H^{10d}(\mathbb{R}^d)$ satisfy $||f_0||_{H^{10d}(\mathbb{R}^d)}\leq c_0$. There exists $T=T(c_0)$ such that for any $\mu_1,\mu_2\in (0,1]$, the Cauchy problem \eqref{appMuskat} has a unique classical solution $f$ in $(0,T]$ satisfying $f(t)\in C^\infty(\mathbb{R}^d)\cap L^2(\mathbb{R}^d)$ for any $t\in (0,T]$. In particular, 
	\begin{align*}
		\sup_{t\in [0,T]}||f(t)||_{H^{10d}}\leq C(c_0).
	\end{align*}
	\begin{remark}\label{Linfty}
		Let $x_t$ satisfy $f(t,x_t)=\sup_{x} f(t,x)$. Then we have $\nabla f(x_t)=0$, $\Delta f(x_t)\leq 0$ and $\delta_\alpha f(x_t)>0$. From \eqref{appMuskat} one has $\frac{df(t,x_t)}{dt}\leq 0$. A similar argument holds for $\inf_{x} f(t,x)$. Hence the $L^\infty$ norm of the solution is non-increasing. By standard interpolation, $\|f\|_{L^\infty}$ can be controlled by $\|f_0\|_{L^2}$ and $\|\nabla f_0\|_{L^\infty}$.
	\end{remark}
\end{proposition}

\subsection{Notations}
From Proposition \ref{Cauchyofapp}, there exists a  smooth solution of the equation \eqref{appMuskat} in $[0,T^\star]$ with smooth initial data $f_{0,2}$, we denote it by $F_2$. We also denote $f$ as the solution with initial data $f_0$. Without loss of generality, we assume there exists a constant $R>1$ depending on $\sigma$ such that 
\begin{equation}\label{F2smooth}
	\sup_{t\in [0,T^\star]}\|F_2(t)\|_{H^{10d}}\leq R.
\end{equation}
Denote
\[
F_1=f-F_2.
\]
The main idea in this paper is to estimate $F_1$.
For any $t\in [0,T^*]$, denote 
\begin{align}\label{defxtj}
	M_j(t)=\partial_j F_1(t,x_{t,j})=\sup_x \partial_j F_1(t,x),
\end{align}
\[
m_j(t)=-\partial_j F_1(t,\tilde{x}_{t,j})=\sup_x (-\partial_jF_1(t,x)), ~j=1,\cdots,d.
\]
We set 
\begin{equation}\label{defA}
	A(t)=\sum_{j=1}^d(|m_j(t)|+|M_j(t)|),
\end{equation}
\begin{equation}\label{defBj}
	B_j(t)=\int\frac{\Delta_\alpha \partial_jF_1(t,x_{t,j})}{\left\langle \Delta_\alpha f\right\rangle^{d+1}} \left(1-\chi\left(\frac{\alpha}{\mu_2}\right)\right)\frac{d\alpha}{|\alpha|^{d}}.
\end{equation}
The key estimate of the Lipschitz norm is that for any $t\in (0,T^\star]$
\begin{equation}\label{Ke}	
	\frac{dM_j}{dt}+\frac{1}{2}B_j\leq  C_0R^3 (1+A)+C_1AB_j+C R A [\log(2+\|\nabla F_1\|_{\dot C^\frac{1}{2}})],
\end{equation}
where $C_0=C(\| f_0\|_{L^2},\|\nabla f_0\|_{L^\infty})$ and $C_1,C$ are constants depend on dimension $d$. We note that we can estimate $m_j(t)$ in a similar manner. The above estimate is established in Section \ref{sectionLip}.  To control $[\log(2+\|\nabla F_1\|_{\dot C^\frac{1}{2}})]$ in $L^1$ in time, in Section \ref{ImproveREG} we use ideas in \cite{		ThomasHfirst,Alazard2020,Alazard2020endpoint,TH4}  to obtain an energy inequality (see \eqref{z11}) in $H^{s}$ with $s$ large.


We note that without specified, we use $C$ to denote constants only depend on the dimension $d$, the value of $C$ may be different from line to line. We introduce the notation $a\lesssim b$, which means that there exists a constant $C$ such that $a\leq Cb$. We denote $a\lesssim_m b$ if the implicit constant also depends on $m$. 		

We introduce the following elementary inequality which will be used frequently in our estimates.
\begin{lemma}\label{Cm}
	For any $a, b\in\mathbb{R}$ and any $m\geq 1$, we have the following inequalities
	\begin{equation}\label{eqCm}
		\left|\frac{1}{\langle a\rangle^m}-	\frac{1}{\langle b\rangle^ m}\right|+
		\left|\frac{a}{\langle a\rangle^{m+3}}-	\frac{b}{\langle b\rangle^ {m+3}}\right|\lesssim_m|a-b|.
	\end{equation}
\end{lemma}
\begin{proof}
	It is easy to verify that functions $\frac{1}{\langle z\rangle^m}$ and $\frac{z}{\langle z\rangle^{m+3}}$ are both uniformly Lipschitz on $\mathbb{R}$. Then we get the results.
\end{proof}
\vspace{0.5cm}
\section{A priori estimates}\label{secprioriest}
In this section, we establish a priori estimates for smooth solutions of the regularized Muskat equation \eqref{appMuskat}.
\subsection{Estimate of the $L^2$ norm}
\begin{lemma}
	Assume $0<\mu_2\ll \mu_1<1$, let $f$ be a solution of the Cauchy problem \eqref{appMuskat} with initial data $f_0$, then  for any $t>0$, there holds
	\begin{align}\label{L2max}
		\|f(t)\|_{L^2}\leq e^{ t}\|f_0\|_{L^2}.
	\end{align}
\end{lemma}
\begin{proof}
	We refer the readers to \cite{Constantin2010} and \cite{Constantin2013} for the $L^2$  maximum principle of the original Muskat equation. 
	For simplicity, we denote $\tilde\chi_{\mu_2}(\alpha)=1-\chi\left(\frac{\alpha}{\mu_2}\right)$. We can rewrite \eqref{appMuskat} as 
	$$
	\partial_t f(x)-\mu_1\Delta f(x)
	=\int(x-\alpha)\cdot\nabla_x G(\Delta_{x-\alpha}f(x))\tilde\chi_{\mu_2}\left(x-\alpha\right)\frac{d\alpha }{|x-\alpha|^{d}}
	$$
	with $
	G(a)=\int_0^a\frac{1}{\langle s\rangle^{d+1}}ds.
	$\\
	We multiply the above equation by $f$, integrate over $dx$, and use integration by parts to observe
	\begin{equation*}
			\begin{aligned}
			\frac{1}{2}\frac{d}{dt}\|f\|_{L^2}^2+\mu_1\|\nabla f\|_{L^2}^2&=-\iint G(\Delta_{x-\alpha} f)(x-\alpha)\cdot\nabla_x f(x)\tilde\chi_{\mu_2}\left(x-\alpha\right)\frac{d\alpha }{|{x-\alpha}|^d} dx\\
			&\quad\quad+\iint\mu_2^{-1}\chi'\left(\frac{x-\alpha}{\mu_2}\right)G(\Delta_{x-\alpha} f)f(x)\frac{d\alpha }{|{x-\alpha}|^{d-1}} dx\\
			&=:K_1+K_2.
		\end{aligned}
	\end{equation*}
	Define the function $H$  by 
	$
	H(a)=\int_0^aG(s)ds.
	$
	It is easy to observe that 
	\begin{align*}
		K_1=&-\iint(x-\alpha)\cdot\nabla_x H(\Delta_{x-\alpha}f(x))\tilde\chi_{\mu_2}\left(x-\alpha\right)\frac{d\alpha }{|x-\alpha|^{d-1}} dx\\
		&-\iint G(\Delta_{x-\alpha} f(x))\Delta_{x-\alpha} f(x)\tilde\chi_{\mu_2}\left(x-\alpha\right)\frac{d\alpha }{|x-\alpha|^{d-1}} dx.
	\end{align*}
	Integrate by parts we obtain
	\begin{align*}
		K_1=&\iint [H(\Delta_{\alpha}f(x))-G(\Delta_{\alpha} f(x))\Delta_{\alpha} f(x)]\tilde\chi_{\mu_2}\left(\alpha\right)\frac{d\alpha dx}{|\alpha|^{d-1}}\\
		&\quad\quad\quad\quad-\iint\mu_2^{-1}\chi'\left(\frac{\alpha}{\mu_2}\right)H(\Delta_{\alpha} f(x))\frac{d\alpha dx}{|{\alpha}|^{d-2}} .
	\end{align*}
	Since 
	$
	\sup_{s}	\left(H(s)-sG(s)\right)\leq 0$, one gets 
	\[
	K_1\leq -\iint\mu_2^{-1}\chi'\left(\frac{\alpha}{\mu_2}\right)H(\Delta_{\alpha} f(x))\frac{d\alpha }{|{\alpha}|^{d-2}} dx.
	\] 
	Then we have 
	\begin{align*}
		\frac{1}{2}\frac{d}{dt}\|f\|_{L^2}^2+\mu_1\|\nabla f\|_{L^2}^2
		\leq&\iint\mu_2^{-1}\chi'\left(\frac{\alpha}{\mu_2}\right)[G(\Delta_{\alpha} f(x))f(x)-H(\Delta_{\alpha} f(x))]\frac{d\alpha  dx}{|{\alpha}|^{d-2}}.
	\end{align*}
	Note  that $G(s)$ is an odd function, by a change of variable we have
	\begin{align*}
		&\iint\mu_2^{-1}\chi'\left(\frac{\alpha}{\mu_2}\right)G(\Delta_{\alpha} f(x))f(x)\frac{d\alpha dx}{|{\alpha}|^{d-1}}\\&\quad\quad\quad =\frac{1}{2}\iint\mu_2^{-1}\chi'\left(\frac{\alpha}{\mu_2}\right)G(\Delta_{\alpha} f(x))\Delta_{\alpha}f(x)\frac{d\alpha dx}{|{\alpha}|^{d-2}} .
	\end{align*}
	Then we combine $
	\sup_{s}	\left(H(s)-sG(s)\right)\leq 0$ and the fact that $-2\leq\chi'\left(\frac{\alpha}{\mu_2}\right)\leq 0$  to conclude that 
	\begin{align*}
		\frac{1}{2}\frac{d}{dt}\|f\|_{L^2}^2+\mu_1\|\nabla f\|_{L^2}^2&\lesssim\iint_{\mu_2\leq|\alpha|\leq 2\mu_2}\mu_2^{-1}G(\Delta_\alpha f(x))\Delta_\alpha f(x)\frac{d\alpha dx}{|\alpha|^{d-2}} \\
		&\lesssim\mu_2^\frac{1}{2}\iint_{\mu_2\leq|\alpha|\leq 2\mu_2}|\delta_\alpha f(x)|^2\frac{d\alpha dx}{|\alpha|^{d+\frac{3}{2}}} \\
		&\lesssim\mu_2^\frac{1}{2}\|f\|_{\dot H^\frac{3}{4}}^2.
	\end{align*}
	Combining this with Sobolev interpolation inequality and Young's inequality we obtain
	$$
	\frac{1}{2}\frac{d}{dt}\|f\|_{L^2}^2+\mu_1\|f\|_{\dot H^1}^2\leq C\mu_2^\frac{2}{3}\|f\|_{\dot H^1}^2+\frac{1}{2}\|f\|_{L^2}^2.
	$$
	Note that $\mu_2\ll \mu_1$, we can absorb the contribution of $ C\mu_2^\frac{2}{3}\|f\|_{\dot H^1}^2$ by the left hand side, then 
	$$
	\frac{d}{dt}\|f\|_{L^2}^2\leq\|f\|_{L^2}^2.
	$$
	By Gronwall's inequality, we get \eqref{L2max}.
\end{proof}
\vspace{0.5cm}
\subsection{Estimate of  the Lipschitz norm}\label{sectionLip}
In this subsection, we will prove the Lipschitz estimate \eqref{Ke}. \\
For simplicity, we denote $\tilde\chi_{\mu_2}(\alpha)=1-\chi\left(\frac{\alpha}{\mu_2}\right)$ and  $d\eta(\alpha)=\tilde\chi_{\mu_2}(\alpha)\frac{d\alpha}{|\alpha|^{d}}$. We define 
$$E_\alpha f(x)=\hat\alpha\cdot\nabla f(x)-\Delta_\alpha f(x).$$ Assume $t\in(0,T^\star].$ 
By taking one derivative $\partial_j=\partial_{x_j}$ in equation \eqref{appMuskat} we obtain
\begin{equation}\label{eq}
	\begin{aligned}
		\partial_t\partial_jf-\mu_1\partial_j\Delta f=&\int\frac{\hat\alpha\cdot\nabla\partial_j f}{\left\langle \Delta_\alpha f\right\rangle^{d+1}}d\eta(\alpha)-\int\frac{\Delta_\alpha \partial_jf}{\left\langle \Delta_\alpha f\right\rangle^{d+1}}d\eta(\alpha)\\
		&\quad\quad\quad\quad-(d+1)\int\frac{E_\alpha f \Delta_\alpha f\Delta_\alpha\partial_j f}{\left\langle \Delta_\alpha f\right\rangle^{d+3}} d\eta(\alpha),
	\end{aligned}
\end{equation}
where we denote $\hat\alpha=\frac{\alpha}{|\alpha|}$. 
Now we look at the above equation with $x=x_{t,j}$, where $x_{t,j}$ is defined in \eqref{defxtj}. Recall the definition of $B_j$ in \eqref{defBj}, we obtain 
\begin{equation}\label{eqinfty}
	\begin{aligned}
		\frac{d M_j}{dt}+B_j
		\leq&\int\frac{E_\alpha \partial_jF_2}{\left\langle \Delta_\alpha f\right\rangle^{d+1}} d\eta(\alpha)-(d+1)\int\frac{E_\alpha F_2\Delta_\alpha f\partial_j\Delta_\alpha f}{\left\langle \Delta_\alpha f\right\rangle^{d+3}} d\eta(\alpha)\\
		&\quad\quad\quad-(d+1)\int\frac{E_\alpha F_1\Delta_\alpha f\partial_j\Delta_\alpha f}{\left\langle \Delta_\alpha f\right\rangle^{d+3}} d\eta(\alpha)-\frac{dF_2}{dt}+\mu_1\partial_j\Delta F_2.
	\end{aligned}
\end{equation}
where we used the fact that $\nabla\partial_jF_1(t,x_{t,j})=0$ and $-\mu_1\partial_j\Delta F_1(t,x_{t,j})\geq0$. From \eqref{F2smooth} it is easy to check that 
\[
\left|\frac{dF_2(t,x_{t,j})}{dt}\right|+\left|\mu_1\partial_j\Delta F_2(t,x_{t,j})\right|\lesssim R.
\]
We denote the first three terms in \eqref{eqinfty} by $J_1,J_2,J_3$. 
For $J_1$, we have, 
\begin{align*}
	\int_{|\alpha|\leq1}\frac{\left|E_\alpha \partial_jF_2\right|}{\left\langle \Delta_\alpha f\right\rangle^{d+1}}\frac{d\alpha}{|\alpha|^{d}}\lesssim \|F_2\|_{\dot C^3}\int_{|\alpha|\leq1}\frac{d\alpha}{|\alpha|^{d-1}}\lesssim R.
\end{align*}
On the other hand,	for $|\alpha|\geq1$, we have
\begin{align*}
	\left|\int_{|\alpha|\geq1}\frac{\hat{\alpha}\cdot\nabla\partial_j F_2}{\left\langle \Delta_\alpha f\right\rangle^{d+1}} d\eta(\alpha)\right|
	&=\frac{1}{2}\left|\int_{|\alpha|\geq1}\hat{\alpha}\cdot\nabla\partial_j F_2\left(\frac{1}{\left\langle \Delta_\alpha f\right\rangle^{d+1}}-\frac{1}{\left\langle \Delta_{-\alpha} f\right\rangle^{d+1}}\right)d\eta(\alpha)\right|\\
	&\overset{\eqref{eqCm}}\lesssim R\int_{|\alpha|\geq1}|\Delta_\alpha f-\Delta_{-\alpha}f|\frac{d\alpha}{|\alpha|^{d}}\\
	&~~\lesssim R\|f\|_{L^\infty},
\end{align*}
and
\[
\left|\int_{|\alpha|\geq1}\frac{\Delta_\alpha \partial_jF_2}{\left\langle \Delta_\alpha f\right\rangle^{d+1}}d\eta(\alpha)\right|\leq\|F_2\|_{\dot C^1}\int_{|\alpha|>1}\frac{d\alpha}{|\alpha|^{d+1}}\lesssim R.
\]
By Remark \ref{Linfty} we have $\|f\|_{L^\infty}\leq C_0=C(\|f_0\|_{L^2},\|f_0\|_{\dot W^{1,\infty}})$.
Then we get
\begin{equation}\label{firstterm}
	J_1\leq C_0R.
\end{equation}
Similarly, for any $\delta>0$, it is easy to verify that 
\begin{align*}
	|E_\alpha F_2 \Delta_\alpha\partial_j f|
	&\lesssim|\alpha| R(R+\Delta_\alpha\partial_j F_1)\mathbf{1 }_{|\alpha|\leq\delta}+|\alpha|^{-1}R(R+A)\mathbf{1 }_{|\alpha|\geq\delta}.
\end{align*}
Hence we obtain
\begin{align}
	J_2&\lesssim \delta R\int_{|\alpha|\leq \delta}\frac{|\Delta_\alpha\partial_j F_1|}{\left\langle \Delta_\alpha f\right\rangle^{d+1}}\frac{d\eta(\alpha)}{|\alpha|^d}+R^2\int_{|\alpha|\leq  \delta}\frac{d\alpha}{|\alpha|^{d-1}}+R(R+A)\int_{|\alpha|\geq  \delta}\frac{d\alpha}{|\alpha|^{d+1}}\nonumber\\
	&\lesssim \delta RB_j+\delta R^2+\delta^{-1}R(R+A).\label{secondterm}
\end{align}
We can choose $\delta\leq CR^{-1}$ small enough such that the contribution of $\delta RB_j$ can be absorbed by $\frac{1}{2}B_j$ in the left hand side.
For $J_3$ we have
\begin{align}
	\nonumber&J_3\lesssim \int\frac{| E_\alpha F_1 \Delta_\alpha\partial_j F_1|}{\left\langle \Delta_\alpha f\right\rangle^{d+1}} d\eta(\alpha)+\left(\int_{|\alpha|\leq 1}
	+\int_{|\alpha|\geq 1}\right)\frac{ |E_\alpha F_1 \Delta_\alpha  \partial_jF_2|}{\left\langle\Delta_\alpha f \right\rangle^{d+1}} d\eta(\alpha)\\
	&\quad\lesssim  AB_j+R\int_{|\alpha|\leq 1}|E_\alpha F_1| \frac{d\alpha}{|\alpha|^d}+ R A. \label{third term}
\end{align}
Combining \eqref{firstterm}, \eqref{secondterm} and \eqref{third term} we get
\begin{equation}\label{finalest}
	\frac{d M_j}{dt}+\frac{1}{2}B_j\leq  C_0R^3 (1+A)+C_1AB_j+CR\int_{|\alpha|\leq 1}|E_\alpha F_1| \frac{d\alpha}{|\alpha|^d}
\end{equation}
for any $t\in (0,T^\star]$,	where $C_0=C(\| f_0\|_{L^2},\|\nabla f_0\|_{L^\infty})$ and $C_1,C$ are constants depend on dimension $d$.
Combining \eqref{finalest} and Lemma \ref{Interpolation} below we obtain \eqref{Ke}.\vspace{0.1cm}\\
\begin{remark}\label{rem1}
	We note that the Lipschitz estimate in 2D is simpler. In 2D, equation \eqref{eq} reads
	\begin{equation*}
		\begin{aligned}
			\partial_tf_x+\int_{\mathbb{R}}\frac{\delta_\alpha f_x}{\left\langle \Delta_\alpha f\right\rangle^{2}}\frac{d\alpha}{\alpha^2}=&\int_{\mathbb{R}}\frac{f_{xx}}{\left\langle \Delta_\alpha f\right\rangle^{2}}\frac{d\alpha}{\alpha}-2\int_{\mathbb{R}}\frac{E_\alpha f \Delta_\alpha f\delta_\alpha f_x}{\left\langle \Delta_\alpha f\right\rangle^{4}} \frac{d\alpha}{\alpha^2},
		\end{aligned}
	\end{equation*}
	where $\Delta_\alpha f(x)=\frac{f(x)-f(x-\alpha)}{\alpha}$.
	Let $x_t\in\mathbb{R}$ such that $f_x(x_t)=\sup_{x\in\mathbb{R}} f_x(x)$. Then we have $f_{xx}(x_t)= 0$ and $\delta_\alpha f_x(x_t)\geq 0$. If $\|f\|_{Lip}$ is small, then $|E_\alpha f \Delta_\alpha f|$ is small. Hence the last term in the right hand side can be absorbed by the left hand side term, which leads to the desired estimate $\partial_t\|f_x\|_{L^\infty}\leq 0$. We also mention the result of \cite{33}. For large monotone increasing initial data, we have $E_\alpha f(x_t)\geq 0$ and $\Delta_\alpha f(x_t)\geq 0$. Hence we have $\partial_t\|f_x\|_{L^\infty}\leq 0$. Similar arguments hold for monotone decreasing initial data.
\end{remark}
In order to control the remainder term in \eqref{finalest}, we introduce the following lemma
\begin{lemma}\label{Interpolation}
	For any function $g$, we have the following interpolation inequality
	\begin{align*}
		\sup_{x}	\int_{|\alpha|\leq 1} |E_\alpha g(x)|\frac{d\alpha}{|\alpha|^d}\lesssim\|\nabla g\|_{L^\infty}\log(2+\|\nabla g\|_{\dot C^{\frac{1}{2}}})+1.
\end{align*}
\end{lemma}
\begin{proof}
	Indeed, we can further split the integration into $|\alpha|\leq \lambda$ and $\lambda\leq|\alpha|\leq 1$, then
	\begin{align*}
		\text{LHS}\leq\lambda^{\frac{1}{2}}\|\nabla g\|_{\dot C^\frac{1}{2}} +\log (\lambda^{-1})\|\nabla g\|_{L^\infty} \lesssim\|\nabla g\|_{L^\infty}\log(2+\|\nabla g\|_{\dot C^\frac{1}{2}})+1,
	\end{align*}
	where we take $\lambda=(1+\|\nabla g\|_{\dot C^\frac{1}{2}})^{-2}$.
\end{proof}
\vspace{0.5cm}

We observe from the above lemma that, if we want to control the $\dot W^{1,\infty}$ norm of $F_1$, we need to improve the regularity. Before that, we first estimate the $L^\infty$ norm of $\nabla f$. More precisely, we have the following result
\begin{proposition}\label{pro3}
	Assume $||\nabla f_{0,1}||_{L^\infty}\leq \frac{1}{100d(d+1)}$. 	There exists $t_1=t_1(R,||\nabla f_0||_{L^\infty})$ such that   if  $	\sup_{\tau\in [0,T_1]}||\nabla F_1(\tau)||_{L^\infty}\leq 10d||\nabla f_{0,1}||_{L^\infty} $
	with  $0<T_1\leq T^\star$, then 
	\begin{align*}
		\sup_{\tau\in[0,\min\{T_1,t_1\}]}||\nabla f(\tau)||_{L^\infty}\leq 1+2||\nabla f_0||_{L^\infty}.
	\end{align*}
\end{proposition}
\begin{proof}
	Let $y_{t,j}$ satisfy $\partial_j f(t,y_{t,j})=\sup_{x}\partial_j f(t,x)$. Then one has $\nabla\partial_j f(y_{t,j})=0$, $\delta_\alpha\partial_j f>0$, and $\partial_j\Delta f(y_{t,j})\leq0$. We substitute $x=y_{t,j}$ in \eqref{eq}, then
	\begin{align*}
		&\partial_t\partial_j f+\int\frac{\delta_\alpha \partial_jf}{\left\langle \Delta_\alpha f\right\rangle^{d+1}}\frac{ d\eta(\alpha)}{|\alpha|}\leq-(d+1)\int\frac{E_\alpha f\Delta_\alpha f\Delta_\alpha\partial_j f}{\left\langle \Delta_\alpha f\right\rangle^{d+3}} d\eta(\alpha).
	\end{align*}
	We split the integral right hand side into small scales $|\alpha|\leq\kappa$ and large scales $|\alpha|\geq\kappa$ for some $\kappa>0$. As in previous discussions, we have
	\begin{align*}
		|E_\alpha f \Delta_\alpha\partial_j f|&\leq C|\alpha|^{-1}\|\nabla f\|_{L^\infty}^2\mathbf{1}_{|\alpha|\geq \kappa}+(2\|\nabla F_1\|_{L^\infty}+\kappa R)|\Delta_\alpha\partial_jf|\mathbf{1}_{|\alpha|\leq \kappa}.
	\end{align*}
	Hence one obtains for $x=y_{t,j}$
	\begin{align*}
		&\partial_t\partial_j f+\int\frac{\delta_\alpha \partial_jf}{\left\langle \Delta_\alpha f\right\rangle^{d+1}}\frac{ d\eta(\alpha)}{|\alpha|}\\
		&\quad\quad\quad\leq C\kappa^{-1}||\nabla f||_{L^\infty}^2+(20d||\nabla f_{0,1}||_{L^\infty}+\kappa R)(d+1)\int\frac{\delta_\alpha\partial_j f}{\left\langle \Delta_\alpha f\right\rangle^{d+1}} \frac{d\eta(\alpha)}{|\alpha|}.
	\end{align*}
	Let $\kappa=(4(d+1)R)^{-1}$, then $(20d||\nabla f_{0,1}||_{L^\infty}+\kappa R)(d+1)\leq \frac{1}{2}$. Then, the last term of RHS can be absorbed by LHS, one obtains 
	\begin{equation*}
		\partial_t\partial_j f\lesssim 
		||\nabla f||_{L^\infty}^2R.
	\end{equation*}
	Similar arguments hold for $\inf_x\partial_j f(t,x)$. 
	Take sum in $j$ we get
	\begin{equation*}
		\partial_t||\nabla f(t)||_{L^\infty}\leq 
		CR(||\nabla f(t)||_{L^\infty}+1)^2.
	\end{equation*}
	Let $t_1>0$ satisfy 
	$
	CRt_1=1/(2(1+||\nabla f_0||_{L^\infty})).$
	We observe that for any $t\leq \min\{T_1,t_1\}$, one has
	$$
	\int_{0}^{t}\partial_t\left(\frac{-1}{1+||\nabla f(t)||_{L^\infty}}\right)dt=\int_{0}^{t}\frac{\partial_t||\nabla f(t)||_{L^\infty}}{(||\nabla f(t)||_{L^\infty}+1)^2}dt\leq CRt_1.
	$$
	Hence we obtain
	$$
	\frac{1}{1+||\nabla f_0||_{L^\infty}}-\frac{1}{1+||\nabla f(t)||_{L^\infty}}\leq\frac{1}{2(1+||\nabla f_0||_{L^\infty})}.$$
	We get the estimate
	\begin{equation*}
		\sup_{\tau\in[0,\min\{T_1,t_1\}]}	||\nabla f(\tau)||_{L^\infty}\leq 1+2||\nabla f_0||_{L^\infty},
	\end{equation*}
	which completes the proof.
\end{proof}
\vspace{0.5cm}
\section{Improve the regularity}\label{ImproveREG}
This section is devoted to improve the regularity of the solution, which helps us to control the remainder terms in the Lipschitz estimate \eqref{finalest}. The main result is the following proposition
\begin{proposition}\label{propyiran}
	For any $r_0>0$, there exists $\sigma=\sigma(r_0)\in(0,1)$ such that, for any $T\in[0,T^\star]$, if 
	\begin{align*}
		\sup_{\tau\in [0,T]}||\nabla f(\tau)||_{L^\infty}\leq r_0 \quad\quad\text{and}\quad\quad 	\sup_{\tau\in [0,T]}\|\nabla F_1(\tau)\|\leq\sigma,
	\end{align*}
	then
	\begin{align}\label{z12}
		\sup_{t\in [0,T]}t^{2(7d-2)}||\Delta^{2d}f(t)||_{L^2}^2+\int_{0}^{T}s^{2(7d-2)}||\Delta^{2d} f(s)||_{\dot H^{\frac{1}{2}}}^2ds\leq C(r_0,R).
	\end{align}
	In particular, we have
	\begin{equation}\label{eqinthm}
		\sup_{t\in [0,T]}t||\nabla f(t)||_{\dot C^\frac{1}{2}}\leq C(r_0,R),
	\end{equation}
	and 
	\begin{equation}\label{Calpha}
		\int _0^{T} \log(2+\|\nabla F_1(t)\|_{\dot C^\frac{1}{2}}) dt \leq C(r_0,R) T^{\frac{1}{2}}.
	\end{equation}
\end{proposition}
\begin{proof}
	Multiply equation \eqref{appMuskat} by test function $2\Delta^{2d}f$, and substitute $g=f$, $g_1=F_1$, $g_2=F_2$ in
	Proposition \ref{thm1} below,  we obtain that for any $t\in[0,T]$, there holds
	\begin{align}\nonumber
		\frac{d}{dt}||\Delta^{2d}f||_{L^2}^2+2\mu_1\|\Delta^{2d}\nabla f\|_{L^2}^2+&\iint \frac{|\delta_\alpha \Delta^{2d} f(x)|^2}{\left\langle r_0 \right\rangle^{d+1}} \frac{d\eta(\alpha) }{|\alpha|}dx\\
		&\lesssim_{r_0} (\sigma^{\frac{1}{8d-1}}+\varepsilon) ||\Delta^{2d} f||_{\dot H^{\frac{1}{2}}}^2+\mathcal{F}(R+\varepsilon^{-1})\label{z11}
	\end{align}	
	for any $\varepsilon\in (0,1)$,  where we denote $\mathcal{F}:(1,\infty)\to (1,\infty)$  to be some functions increasing and $\mathcal{F}(r)\to \infty$ as $r\to \infty$. The definition of $\mathcal{F}$ may be different from line to line. 
	Note  that 
	\begin{align*}
		\iint \frac{|\delta_\alpha \Delta^{2d} f(x)|^2}{\left\langle r_0 \right\rangle^{d+1}} \frac{d\eta(\alpha) }{|\alpha|}dx
		&=\frac{1}{\langle r_0\rangle^{d+1}}\|\Delta^{2d}f\|_{\dot{F}_{2,2}^\frac{1}{2}}-\iint \frac{|\delta_\alpha \Delta^{2d} f(x)|^2{\chi}\left(\frac{\alpha}{\mu_2}\right)}{\left\langle r_0 \right\rangle^{d+1}} \frac{d\alpha dx}{|\alpha|^{d+1}}.
	\end{align*}
	It is easy to verify that
	\begin{align*}
		\iint \frac{|\delta_\alpha \Delta^{2d} f(x)|^2{\chi}\left(\frac{\alpha}{\mu_2}\right)}{\left\langle r_0 \right\rangle^{d+1}} \frac{d\alpha }{|\alpha|^{d+1}}dx&\leq\mu_2^\frac{1}{2}\int_{|\alpha|\leq\mu_2} \|\delta_\alpha \Delta^{2d} f(x)\|_{L^2}^2\frac{d\alpha }{|\alpha|^{d+\frac{3}{2}}}\\
		&\lesssim \mu_2^\frac{1}{2}\|\Delta^{2d} f(t)\|_{\dot H^{\frac{1}{2}}}\|\Delta^{2d} f(t)\|_{\dot H^{1}}\\
		&\lesssim \mu_2^\frac{1}{2}\|\Delta^{2d} f(t)\|_{\dot H^{\frac{1}{2}}}^2+\mu_2^\frac{1}{2}\|\Delta^{2d} f(t)\|_{\dot H^{1}}^2,
	\end{align*}
	where we also used H\"{o}lder's inequality and Young's inequality. Note that $\mu_2\ll\mu_1$, hence the contribution of $\mu_2^\frac{1}{2}\|\Delta^{2d} f(t)\|_{\dot H^{1}}^2$ can be absorbed by the viscosity terms.
	Then \eqref{z11} leads to 
	$$
	\frac{d}{dt}||\Delta^{2d}f||_{L^2}^2+\frac{1}{\langle r_0\rangle^{d+1}}||\Delta^{2d} f||_{\dot H^{\frac{1}{2}}}^2\leq C(r_0) (\sigma^{\frac{1}{8d-1}}+\varepsilon+\mu_2^\frac{1}{2}) ||\Delta^{2d} f||_{\dot H^{\frac{1}{2}}}^2
	+\mathcal{F}(r_0+R+\varepsilon^{-1}).
	$$
	We can take $\sigma, \varepsilon, \mu_2$ small enough such that
	$$
	C(r_
	0)(\sigma^{\frac{1}{8d-1}}+\varepsilon+\mu_2^\frac{1}{2})\leq \frac{1}{2\langle r_0\rangle^{d+1}}.
	$$
	Then
	$$
	\partial_t||\Delta^{2d}f||_{L^2}^2+\frac{1}{2\langle r_0\rangle^{d+1}}||\Delta^{2d} f||_{\dot H^{\frac{1}{2}}}^2\leq \mathcal{F}(r_0+R)
	$$
	for any $t\leq T$.  For any $0\leq s<t\leq T$, integrate the above inequality in time we get
	$$
	||\Delta^{2d}f(t)||_{L^2}^2+\frac{1}{2\langle r_0\rangle^{d+1}}\int_{s}^{t}||\Delta^{2d} f(\tau)||_{\dot H^{\frac{1}{2}}}^2 d\tau\lesssim ||\Delta^{2d}f(s)||_{L^2}^2+\mathcal{F}(r_0+R) T.
	$$
	Then we multiply the above equation by $s^{m-1}$ and integrate for $s\in[0,t]$ to get
	\begin{align}
		\sup_{t\in [0,T]}t^{m}||\Delta^{2d}f(t)||_{L^2}^2+&\frac{1}{2\langle r_0\rangle^{d+1}}\int_{0}^{T}s^m||\Delta^{2d} f||_{\dot H^{\frac{1}{2}}}^2ds\nonumber\\
		&\leq  C\int_{0}^{T}s^{m-1}||\Delta^{2d}f(s)||_{L^2}^2ds+\mathcal{F}(r_0+R) T^{m+1}.\label{ineqintertime}
	\end{align}
	Applying the standard Sobolev interpolation inequality and the $L^2$ estimates \eqref{L2max}, one has
	\[
	||\Delta^{2d}f(s)||_{L^2}^2\lesssim ||\Delta^{2d}f(s)||_{\dot H^{\frac{1}{2}}}^{\frac{16d}{8d+1}}||f(s)||_{L^2}^{\frac{2}{8d+1}}\lesssim||\Delta^{2d}f(s)||_{\dot H^{\frac{1}{2}}}^{\frac{16d}{8d+1}}||f_0||_{L^2}^{\frac{2}{8d+1}}.
	\]
	By H\"{o}lder's inequality and Young's inequality,
	\begin{equation*}
		\int_{0}^{T}s^{m-1}||\Delta^{2d}f(s)||_{\dot H^{\frac{1}{2}}}^{\frac{16d}{8d+1}}ds\leq \frac{\int_{0}^{T}s^m||\Delta^{2d} f(s)||_{\dot H^{\frac{1}{2}}}^2ds}{4\langle r_0\rangle^{d+1}}+C(r_0)\int_0^{T}s^{m-8d-1}ds,
	\end{equation*}
	Note that the first term can be absorbed by the left hand side of \eqref{ineqintertime}. 
	To make the last term finite, we choose  $m=2(7d-2)>8d$, which leads to \eqref{z12}.
	
	The Gagliardo-Nirenberg inequality implies that
	$$
	||\nabla f(t)||_{\dot C^\frac{1}{2}}\lesssim||\Delta^{2d}f(t)||_{L^2}^{\frac{1}{7d-2}}||\nabla f(t)||_{L^\infty}^{\frac{7d-3}{7d-2}}\lesssim_{r_0}||\Delta^{2d}f(t)||_{L^2}^{\frac{1}{7d-2}}.
	$$
	Combining this with \eqref{z12} and the definition of $R$ in \eqref{F2smooth}, we obtain 
	\eqref{eqinthm} and \eqref{Calpha}. This completes the proof.
\end{proof}
Denote 
\begin{align}\label{defN}
	N(f,g)=\int\frac{\hat\alpha\cdot\nabla f(x)-\Delta_\alpha f(x)}{\left\langle \Delta_\alpha g\right\rangle^{d+1}}d\eta(\alpha).
\end{align}
We have the following proposition
\begin{proposition}\label{thm1} For any function $g=g_1+g_2$ with  $||g_1||_{Lip}\leq 1$ and $||\nabla g||_{L^\infty}\leq r_0$, there holds
	\begin{align*}\nonumber
		&\left|\iint\Delta^{2d}(N(g,g))\Delta^{2d} g(x)dx+\frac{1}{2}\iint \frac{|\delta_\alpha \Delta^{2d} g(x)|^2}{\left\langle \hat \alpha \cdot \nabla g(x) \right\rangle^{d+1}} \frac{d\eta(\alpha) }{|\alpha|}dx\right|\\&\quad\qquad\quad\quad\qquad\quad\lesssim_{r_0}  (||\nabla g_1||_{L^\infty}^{\frac{1}{8d-1}}+\varepsilon) ||\Delta^{2d} g||_{\dot H^{\frac{1}{2}}}^2+\mathcal{F}(||g_2||_{ H^{4d+\frac{1}{2}}}+\varepsilon^{-1})
	\end{align*}	
	for any $\varepsilon\in (0,1)$.
\end{proposition}
We postpone the proof in the Appendix.

\section{Complete the Proof of Theorem \ref{mainthm1} }\label{sectcomplete}
\begin{proof} Assume $||\nabla f_{0,1}||_{L^\infty}\leq \frac{1}{800 d^2(C_1+1)}$, where $C_1$ is the constant in the inequality \eqref{Ke}. Let $T_1\leq T^\star$ be such that  
	\begin{equation}
		\label{xu1}\sup_{\tau\in [0,T_1]}||\nabla F_1(\tau)||_{L^\infty}\leq 10d||\nabla f_{0,1}||_{L^\infty}.
	\end{equation} By Proposition \ref{pro3}, there exists $t_1$ independent of $T_1$ such that
	\begin{align}\label{resultforf}
		\sup_{\tau\in [0,\min\{T_1,t_1\}]}	||\nabla f(\tau)||_{L^\infty}\leq 1+2||\nabla f_0||_{L^\infty}.
	\end{align}
	By the Lipschitz estimate \eqref{Ke}, we have for any $0\leq t\leq T_1$
	\begin{align*}
		\frac{dM_j}{dt}+\frac{1}{2}B_j&\leq  C_0R^3 (1+A)+C_1AB_j+C R A \log(2+\|\nabla F_1\|_{\dot C^\frac{1}{2}})\\ &\leq  2C_0R^3 +\frac{1}{10}B_j+C R \log(2+\|\nabla F_1\|_{\dot C^\frac{1}{2}}).
	\end{align*}
	Here we have used the definition of $A$ in \eqref{defA} to get 
	\begin{align*}
		A(t)\leq 2d ||\nabla F_1(t)||_{L^\infty}\leq 20d^2||\nabla f_{0,1}||_{L^\infty}\leq \frac{1}{40 (C_1+1)}.
	\end{align*}
	So, for any $0\leq t\leq T_1$,
	\begin{align*}
		\frac{dM_j}{dt}\leq   2C_0R^3+C R \log(2+\|\nabla F_1\|_{\dot C^{\frac{1}{2}}}).
	\end{align*}
	Note that the same arguments are valid for $m_j(t)$. Recalling the definition \eqref{defA}
	and summing in $j$ we have 
	\begin{align*}
		\frac{dA}{dt}\leq   4dC_0R^3+C R \log(2+\|\nabla F_1\|_{\dot C^\frac{1}{2}}).
	\end{align*}
	This implies that for any $0\leq t\leq T_1$,
	\begin{align}\nonumber
		\sup_{\tau\in[0,t]} ||\nabla F_1(\tau)||_{L^\infty}&\leq \sup_{\tau\in[0,t]} A(\tau)\\&\nonumber\leq A(0)+4dC_0R^3t+C R \int_{0}^{t}\log(2+\|\nabla F_1(\tau)\|_{\dot C^\frac{1}{2}})d\tau\\& \leq  2d||\nabla f_{0,1}||_{L^\infty}+4dC_0R^3t+C R \int_{0}^{t}\log(2+\|\nabla F_1(\tau)\|_{\dot C^\frac{1}{2}})d\tau \label{xu2}.
	\end{align} 
	Let $r_0=1+2||\nabla f_0||_{L^\infty}$ and $\sigma_1$ be in Proposition \ref{propyiran}  associated to $r_0$. Assume 
	$||\nabla f_{0,1}||_{L^\infty}\leq \frac{\sigma_1}{800 d(C_1+1)}$. Then by \eqref{xu1}
	$$	\sup_{\tau\in [0,T_1]}||\nabla F_1(\tau)||_{L^\infty}\leq 10d||\nabla f_{0,1}||_{L^\infty}\leq \sigma_1.$$
	Now we can apply  Proposition \ref{propyiran} to $T=t\leq \min\{T_1,t_1\}$ and obtain 
	\begin{equation*}
		\int _0^{t} \log(2+\|\nabla F_1(\tau)\|_{\dot C^\frac{1}{2}}) d\tau \leq C(r_0,R) t^{\frac{1}{2}}.
	\end{equation*}
	Combining this with \eqref{xu2} yields 
	\begin{align}\nonumber
		\sup_{\tau\in[0,t]} ||\nabla F_1(\tau)||_{L^\infty}\leq  2d||\nabla f_{0,1}||_{L^\infty}+4dC_0R^3t+C(r_0,R) t^{\frac{1}{2}}.
	\end{align} 
	Set $$t_2= \frac{||\nabla f_{0,1}||^2_{L^\infty}}{4\left(C(r_0,R)+R^2+10d(C_0+1)\right)^2},$$
	then we have
	$$
	4dC_0R^3t_2+C(r_0,R) t_2^{\frac{1}{2}}\leq ||\nabla f_{0,1}||_{L^\infty}.
	$$
	Thus, 
	\begin{equation}\label{xu4}
		\sup_{\tau\in[0,\min\{T_1,t_1,t_2\}]} ||\nabla F_1(\tau)||_{L^\infty}\leq  (2d+1)||\nabla f_{0,1}||_{L^\infty}.
	\end{equation}
	Now we will prove that 
	\begin{equation*}
		\sup_{\tau\in[0,\min\{T^\star,t_1,t_2\}]} ||\nabla F_1(\tau)||_{L^\infty}\leq  10d||\nabla f_{0,1}||_{L^\infty}.
	\end{equation*}
	In fact, set 
	\begin{align*}
		\tau_0=\sup\left\{t\in [0,\min\{T^\star,t_1,t_2\}] :\sup_{\tau\in [0,t]}|\nabla F_1(\tau)||_{L^\infty}\leq  10d||\nabla f_{0,1}||_{L^\infty}\right\}.
	\end{align*}
	If $\tau_0<\min\{T^\star,t_1,t_2\}$, using \eqref{xu1} and\eqref{xu4} with $T_1=\tau_0$, we have
	\begin{align*}	\sup_{\tau\in[0,\tau_0]} ||\nabla F_1(\tau)||_{L^\infty}\leq  (2d+1)||\nabla f_{0,1}||_{L^\infty}<10d||\nabla f_{0,1}||_{L^\infty},
	\end{align*}
	which  contradicts the definition that $\tau_0$ is a  
	supremum. Hence $\tau_0=\min\{T^\star,t_1,t_2\}$. Then \eqref{resultforf} implies
	\begin{align*}
		\sup_{t\in[0,\tau_0]}\|\nabla f\|_{L^\infty}\leq 1+2\|\nabla f_0\|_{L^\infty}.
	\end{align*}
	Combining this with  \eqref{eqinthm} and the standard compactness argument, we are able to pass the limit $\mu_2\rightarrow0$ and then $\mu_1\rightarrow 0$ to get a solution of the Cauchy problem \eqref{E1}, which also satisfy the above estimates. Thus we complete the proof of Theorem \ref{mainthm1} with $\sigma=\frac{\sigma_1}{800 d^2(C_1+1)}$.
\end{proof}	
\section{Uniqueness}\label{sectionunique}
In this section, we give a proof of Proposition \ref{propunique}.\\
\begin{proof}
	Set $g=f-\bar{f}$, then we have
	$$
	\partial_t g=\int\frac{E_\alpha g}{\left\langle \Delta_\alpha f\right\rangle^{d+1}}\frac{d\alpha}{|\alpha|^d}+\int E_\alpha \bar{f}\left(\frac{1}{\left\langle \Delta_\alpha f\right\rangle^{d+1}}-\frac{1}{\left\langle \Delta_\alpha \bar{f}\right\rangle^{d+1}}\right)\frac{d\alpha}{|\alpha|^{d}}.
	$$
	From Lemma \ref{Cm}, we have
	$$
	\partial_t g\leq \int\frac{E_\alpha g}{\left\langle \Delta_\alpha f\right\rangle^{d}}\frac{d\alpha}{|\alpha|^{d+1}}+C\int\left|E_\alpha \bar{f}\right||\Delta_\alpha g|\frac{d\alpha}{|\alpha|^{d}}.
	$$
	Let $x_t$ satisfy $\mathbf{g}(t):=g(t, x_t)=\sup_{x}g(t,x)$. Then we have $\nabla g(x_t)=0$ and $\delta_\alpha g(x_t)\geq0$. Hence we have for $x=x_t$
	$$
	\frac{d \mathbf{g}}{dt}+\tilde C\int|\delta_\alpha g|\frac{d\alpha}{|\alpha|^{d+1}}\leq C\int\left|E_\alpha \bar{f}\delta_\alpha g\right|\frac{d\alpha}{|\alpha|^{d+1}}
	$$
	where 
	$
	\tilde C=1/(\langle ||\nabla f||_{L^\infty_{t,x}}\rangle^{d+1}).
	$\vspace{0.1cm}\\
	Recall that $\bar f$ can be decomposed that
	$
	\bar{f}=	\bar{f}_1+	\bar{f}_2
	$
	with 
	$
	||\bar{f}_1||_{\dot W^{1,\infty}}\leq \sigma $ and $\bar f_2\in L^\infty([0,T], H^{10d}).
	$\\
	Then for $\sigma\leq\frac{\tilde C}{8C}$, we have for $x=x_t$
	\begin{align*}
		\frac{d \mathbf{g}}{dt}+\frac{\tilde C}{2}\int\frac{\delta_\alpha g}{|\alpha|^{d+1}}d\alpha&\lesssim \epsilon_0\int_{|\alpha|\leq \epsilon_0}|\delta_\alpha g|\frac{d\alpha}{|\alpha|^{d+1}}+||g||_{L^\infty}\int_{|\alpha|\geq \epsilon_0} \left|E_\alpha \bar{f}_2\right|\frac{d\alpha}{|\alpha|^{d+1}}
		\\&\lesssim \epsilon_0\int_{|\alpha|\leq \epsilon_0}|\delta_\alpha g|\frac{d\alpha}{|\alpha|^{d+1}}+\epsilon_0^{-1}||g||_{L^\infty}||\nabla\bar{f}_2||_{L^\infty_{t,x}}.
	\end{align*}
	Choosing $\epsilon_0>0$ small enough, the first term on the right hand side can be absorbed by the left hand side. Then one has
	$$
	\frac{d \mathbf{g}}{dt}\leq C(||\nabla\bar{f}_2||_{L^\infty_{t,x}},||\nabla f||_{L^\infty_{t,x}}) ||g||_{L^\infty}.
	$$
	Replace $g$ by $-g$, a similar discussion shows that the estimate holds for $\inf_xg(t,x)$. Thus we can conclude that
	$$
	\frac{d}{dt}||g(t)||_{L^\infty}\leq C(||\nabla\bar{f}_2||_{L^\infty_{t,x}},||\nabla f||_{L^\infty_{t,x}}) ||g(t)||_{L^\infty}.
	$$
	Finally, combining this with Gronwall's inequality we get Proposition \ref{propunique}.
\end{proof}
\section{Appendix}
In this section, we will prove Proposition \ref{thm1}.
We first review some elementary results about Triebel-Linzorkin spaces following Triebel \cite{triebel}.
\begin{definition} (Triebel-Lizorkin norms)\\
	For any integer $m\geq 0$, and real number $s\in(m,m+1)$ and $p,q\in[1,\infty)$, the homogeneous Triebel-Lizorkin space $\dot F^s_{p,q}(\mathbb{R}^d)$ consists of those tempered distributions $f$ whose Fourier transform is integrable on a neighborhood of the origin and such that
	\[
	\|f\|_{\dot F^s_{p,q}(\mathbb{R}^d)}=\left(\int_{\mathbb{R}^d}\left(\int_{\mathbb{R}^d}|\delta_\alpha D^mf(x)|^q\frac{d\alpha}{|\alpha|^{d+q(s-m)}}\right)^\frac{p}{q}dx\right)^\frac{1}{p}< +\infty.
	\]
	We also define
	\[
	\|f\|_{\dot F^s_{p,\infty}(\mathbb{R}^d)}=\left(\int_{\mathbb{R}^d}\left(\sup_{\alpha}\frac{|\delta_\alpha D^mf(x)|}{|\alpha|^{s-m}}\right)^pdx\right)^\frac{1}{p}.
	\]
	We note that $\|\cdot\|_{\dot H^s}$ and $\|\cdot\|_{\dot F^s_{2,2}}$ are equivalent. Moreover, for any $2<p<\infty$ and  $0<q\leq\infty$, we have
	\[
	\|f\|_{\dot F^s_{p,q}(\mathbb{R}^d)}\lesssim\|f\|_{\dot H^r (\mathbb{R}^d) }\quad\quad\text{for} ~~r=s-\frac{d}{p}+\frac{d}{2}.
	\]
\end{definition}
More generally, we introduce the following Gagliardo-Nirenberg interpolation inequality for the Triebel-Lizokin spaces:\\
Let $1<q\leq\infty$ and  $s>0$. There holds
\begin{align}\label{GNinterpolation}
	\|f\|_{\dot F^{\theta s}_{\frac{2}{\theta},q}}\lesssim\|f\|_{\dot H^{s}}^\theta\|f\|_{L^\infty}^{1-\theta}
\end{align}
for any $\theta\in(0,1)$.
\vspace{0.5cm}\\
To prove Proposition \ref{thm1}, we first observe that
\begin{equation}\label{coro}
	\begin{aligned}
		&\left|\iint\Delta^{2d}(N(g,g))\Delta^{2d} g(x)d\eta(\alpha)dx+\frac{1}{2}\iint \frac{|\delta_\alpha \Delta^{2d} g(x)|^2}{\left\langle \hat \alpha \cdot \nabla g(x) \right\rangle^{d+1}} \frac{d\eta(\alpha) }{|\alpha|}dx\right|\\
		&\quad\lesssim\left|\iint\frac{E_\alpha (\Delta^{2d} g)(x)}{\langle \hat \alpha\cdot \nabla g(x) \rangle^{d+1}}d\eta(\alpha)\Delta^{2d} g(x)dx+\frac{1}{2}\iint \frac{|\delta_\alpha \Delta^{2d} g(x)|^2}{\left\langle \hat \alpha \cdot \nabla g(x) \right\rangle^{d+1}} \frac{d\eta(\alpha) }{|\alpha|}dx\right|\\
		&\quad\quad\quad\quad+ \left|\iint\hat\alpha\cdot\nabla_x V (\alpha,x)d\eta(\alpha)| \Delta^{2d} g(x)|^2 dx\right|+\left|\iint M(\alpha,x)d\eta(\alpha)\Delta^{2d} g(x) dx\right|\\
		&\quad\quad\quad\quad+\left|\iint\Delta_\alpha (\Delta^{2d} g)(x)V(\alpha,x)d\eta(\alpha)\Delta^{2d} g(x) dx\right|\\
		&\quad=:I_1+I_2+I_3+I_4,
	\end{aligned}
\end{equation}
where 
\begin{equation}\label{defM}
	M(\alpha,x)=\Delta^{2d}\left(\frac{E_\alpha g(x)}{\left\langle \Delta_\alpha g(x)\right\rangle^{d+1}}\right)-\frac{E_\alpha(\Delta^{2d} g)(x)}{\left\langle \Delta_\alpha g(x)\right\rangle^{d+1}},
\end{equation}
and 
\begin{equation}\label{defV}
	V(\alpha,x)=\frac{1}{\left\langle \Delta_\alpha g(x)\right\rangle^{d+1}}-\frac{1}{\left\langle \hat \alpha\cdot\nabla g(x) \right\rangle^{d+1}}.
\end{equation}
In fact, recall the definition of $N$ in \eqref{defN},  direct calculation leads to 
\begin{align}
	\Delta^{2d}(N(g,g))
	=&\int \frac{E_\alpha (\Delta^{2d} g)(x)}{\langle \hat \alpha\cdot \nabla g(x) \rangle^{d+1}}d\eta(\alpha)+\int E_\alpha (\Delta^{2d} g)(x)V(\alpha,x)d\eta(\alpha)\nonumber\\
	&\quad\quad+\int M(\alpha,x)d\eta(\alpha).\label{impdeconm}
\end{align}
Note that 
\begin{align*}
	&\int E_\alpha (\Delta^{2d} g)(x)V(\alpha,x)d\eta(\alpha)\\
	&\quad\quad\quad=\int\hat\alpha\cdot\nabla \Delta^{2d} g(x)V(\alpha,x)d\eta(\alpha)
-\int\Delta_\alpha (\Delta^{2d} g)(x)V(\alpha,x)d\eta(\alpha).
\end{align*}
Take the $L^2$ inner product of \eqref{impdeconm} with $\Delta^{2d}g$ and integrate by parts, one has
\begin{align*}
	&\iint\Delta^{2d}(N(g,g))\Delta^{2d} gdx\\
	&=\iint\frac{E_\alpha (\Delta^{2d} g)(x)}{\langle \hat \alpha\cdot \nabla g(x) \rangle^{d+1}}d\eta(\alpha)\Delta^{2d} g(x)dx-\frac{1}{2} \iint\hat\alpha\cdot\nabla_x V (\alpha,x)d\eta(\alpha) | \Delta^{2d} g(x)|^2 dx\\
	&\quad-\iint\Delta_\alpha (\Delta^{2d} g)(x)V(\alpha,x)d\eta(\alpha)\Delta^{2d} g(x)dx +\iint M(\alpha,x)d\eta(\alpha)\Delta^{2d} g(x) dx,
\end{align*}
which leads to \eqref{coro}.
\vspace{0.5cm}\\
To simplify the notations, we set 
$$
p_1=\frac{8d-1}{4d-1},\quad\  p_2=\frac{32d-4}{16d-3}, \quad\  p_3=\frac{32d-4}{3},\quad\  	p_4=16d-2, \quad\  p_5=\frac{4(8d-1)}{5}.
$$
We recall that $\mathcal{F}:(1,\infty)\to (1,\infty)$  denote some increasing functions  and $\mathcal{F}(r)\to \infty$ as $r\to \infty$. The definition of $\mathcal{F}$ may be different from line to line. 
To prove Proposition \ref{thm1}, it remains to estimate the four terms in the right hand side of \eqref{coro}. We finish this in Lemma \ref{le1}$-$Lemma \ref{le3}.
\begin{lemma}	\label{le1} (Estimate for $I_1$)\\
	Let $g$, $r_0$ as defined in Proposition \ref{thm1}, there holds
	\begin{align}\nonumber
		&\left|\iint\frac{E_\alpha (\Delta^{2d} g)(x)}{\langle \hat \alpha\cdot \nabla g(x) \rangle^{d+1}}d\eta(\alpha)\Delta^{2d} g(x)dx+\frac{1}{2}\iint \frac{|\delta_\alpha \Delta^{2d} g(x)|^2}{\left\langle \hat \alpha \cdot \nabla g(x) \right\rangle^{d+1}} \frac{d\eta(\alpha) }{|\alpha|}dx\right|\\
		&\qquad\qquad\qquad\qquad\lesssim_{r_0}  (||\nabla g_1||_{L^\infty}^\frac{16d-5}{16d-2}+\varepsilon) ||\Delta^{2d} g||_{\dot H^{\frac{1}{2}}}^2+\mathcal{F}(||g_2||_{ H^{4d+\frac{1}{2}}}+\varepsilon^{-1})\label{z5}
	\end{align}
	for any $\varepsilon\in (0,1)$.
\end{lemma}
\begin{proof} Set  $h=\Delta^{2d} g$. 	Note  that $\int\frac{\hat\alpha\cdot\nabla h(x)}{\left\langle \hat \alpha \cdot \nabla g(x) \right\rangle^{d+1}}d\eta(\alpha) =0$, hence
	\begin{align*}
		\iint\frac{E_\alpha h(x)}{\langle \hat \alpha\cdot \nabla g(x) \rangle^{d+1}}d\eta(\alpha)h(x)dx&=-\frac{1}{2}\iint \Delta_\alpha h(x)\delta_\alpha\left(\frac{ h(\cdot)}{\left\langle \hat \alpha \cdot \nabla g(\cdot) \right\rangle^{d+1}}\right)(x)d\eta(\alpha) dx.
	\end{align*}
	Then one has
	\begin{align*}
		|I_1|&=\left|	\iint\frac{E_\alpha h(x)}{\langle \hat \alpha\cdot \nabla g(x) \rangle^{d+1}}d\eta(\alpha)h(x)dx+\frac{1}{2}\iint \frac{|\delta_\alpha h(x)|^2}{\left\langle \hat \alpha \cdot \nabla g(x) \right\rangle^{d+1}} \frac{d\eta(\alpha) }{|\alpha|}dx\right|\\
		&\lesssim\left|\iint \delta_\alpha h(x)h(x-\alpha)\delta_\alpha\left(\frac{ 1}{\left\langle \hat \alpha \cdot \nabla g(\cdot) \right\rangle^{d+1}}\right)(x) \frac{d\eta(\alpha) }{|\alpha|}dx\right|.
	\end{align*}
	By Lemma \ref{Cm},
	we obtain
	\begin{align*}
		|I_1|&\lesssim\iint |\delta_\alpha h(x)||h(x-\alpha)||\delta_\alpha\nabla g(x)| \frac{d\alpha dx}{|\alpha|^{d+1}}\\
		&\lesssim\iint |\delta_\alpha h(x)||\delta_\alpha\nabla g(x)||h(x)| \frac{d\alpha dx}{|\alpha|^{d+1}}+\iint |\delta_\alpha h(x)|^2|\delta_\alpha\nabla g(x)| \frac{d\alpha dx}{|\alpha|^{d+1}}\\
		&:=I_{1,1}+I_{1,2}.
	\end{align*}
	For the first term $I_{1,1}$, we apply the H\"{o}lder's inequality to get
	$$
	I_{1,1}
	\leq ||\Delta^{2d}g||_{L^{p_1}} ||\Delta^{2d}g||_{\dot F^{\frac{1}{4}}_{p_2,2}} \left(||\nabla g_1||_{\dot F^{\frac{3}{4}}_{p_3,2}}+||\nabla g_2||_{\dot F^{\frac{3}{4}}_{p_3,2}}\right).
	$$
	We apply H\"{o}lder's inequality again to $I_{1,2}$, then
	\begin{align*}
		I_{1,2}&\leq\iint |\delta_\alpha h(x)|^2|\delta_\alpha\nabla g_1(x)| \frac{d\alpha dx}{|\alpha|^{d+1}}+\iint |\delta_\alpha h(x)|^2|\delta_\alpha\nabla g_2(x)| \frac{d\alpha dx}{|\alpha|^{d+1}}\\
		&\lesssim||\nabla g_1||_{L^\infty}||\Delta^{2d} g||_{\dot H^{\frac{1}{2}}}^2+||\Delta^{2d} g||^2_{\dot F ^{\frac{1}{4}}_{p_2,4}}||\nabla g_2||_{\dot F^{\frac{1}{2}}_{p_4,2}}.
	\end{align*}
	By the interpolation inequality \eqref{GNinterpolation}, one obtains for any $1< q\leq+\infty$
	\begin{equation}\label{GNineqs}
		\begin{aligned}
			||\Delta^{2d}g||_{L^{p_1}} &\lesssim ||\nabla g||_{L^\infty}^{\frac{1}{8d-1}} || g||_{\dot H^{4d+\frac{1}{2}} }^{\frac{8d-2}{8d-1}},\quad\quad
			||\Delta^{2d}g||_{\dot F^{\frac{1}{4}}_{p_2,q}}
			\lesssim ||\nabla g||_{L^\infty}^{\frac{1}{16d-2}} || g||_{\dot H^{4d+\frac{1}{2}} }^{\frac{16d-3}{16d-2}},\\
			||\nabla g||_{\dot F^{\frac{3}{4}}_{p_3,q}}&\lesssim ||\nabla g||_{L^\infty}^{\frac{16d-5}{16d-2}} || g||_{\dot H^{4d+\frac{1}{2}} }^{\frac{3}{16d-2}},\quad\quad
			\|\nabla g\|_{\dot F^{\frac{1}{2}}_{p_4,q}}\lesssim\|\nabla g\|_{L^\infty}^\frac{8d-2}{8d-1}\|\Delta^{2d}g\|_{\dot H^\frac{1}{2}}^\frac{1}{8d-1}.
		\end{aligned}
	\end{equation}
	Note that $\|\nabla g\|_{L^\infty}\leq r_0$, thus we have
	\begin{align*}
		I_{1,1}&\lesssim
		||\nabla g||_{L^\infty}^{\frac{3}{16d-2}}\|g\|_{\dot H^{4d+\frac{1}{2}}}^\frac{32d-7}{16d-2}(\|\nabla g_1\|_{L^\infty}^\frac{16d-5}{16d-2}\|g_1\|_{\dot H^{4d+\frac{1}{2}}}^\frac{3}{16d-2}+\|\nabla g_2\|_{L^\infty}^\frac{16d-5}{16d-2}\|g_2\|_{\dot H^{4d+\frac{1}{2}}}^\frac{3}{16d-2})\\
		&\lesssim_{r_0}\|g\|_{\dot H^{4d+\frac{1}{2}}}^\frac{32d-7}{16d-2}\left(\|\nabla g_1\|_{L^\infty}^\frac{16d-5}{16d-2}(\|g\|_{\dot H^{4d+\frac{1}{2}}}^\frac{3}{16d-2}+\|g_2\|_{\dot H^{4d+\frac{1}{2}}}^\frac{3}{16d-2})+\|\nabla g_2\|_{L^\infty}^\frac{16d-5}{16d-2}\|g_2\|_{\dot H^{4d+\frac{1}{2}}}^\frac{3}{16d-2}\right)\\
		&\lesssim_{r_0}\|g\|_{\dot H^{4d+\frac{1}{2}}}^2\|\nabla g_1\|_{L^\infty}^\frac{16d-5}{16d-2}+\|g\|_{\dot H^{4d+\frac{1}{2}}}^\frac{32d-7}{16d-2}\|g_2\|_{\dot H^{4d+\frac{1}{2}}}^\frac{3}{16d-2},
	\end{align*}
	and 
	\begin{align*}
		I_{1,2}
		&\lesssim||\nabla g_1||_{L^\infty}||\Delta^{2d} g||_{\dot H^{\frac{1}{2}}}^2+ || g||_{\dot H^{4d+\frac{1}{2}} }^{\frac{16d-3}{8d-1}}\|\Delta^{2d}g_2\|_{\dot H^\frac{1}{2}}^\frac{1}{8d-1}.
	\end{align*}
	Combining the above results with Young's inequality we obtain
	\begin{align*}
		I_{1}
		\leq&\left(\|\nabla g_1\|_{L^\infty}^\frac{16d-5}{16d-2}+\varepsilon\right)\|g\|_{\dot H^{4d+\frac{1}{2}}}^2+\mathcal{F}(||g_2||_{ H^{4d+\frac{1}{2}}}+\varepsilon^{-1})
	\end{align*}
	for any $\varepsilon\in (0,1)$.
\end{proof}
\vspace{0.5cm}
\begin{lemma} \label{le2}(Estimate for $I_2$)\\
	Let $g$, $r_0$ as defined in Proposition \ref{thm1}, and $V$ as defined in \eqref{defV}, there holds
	\begin{equation*}
		\begin{aligned}
		&\left|\iint_{\mathbb{R}^d}\hat\alpha\cdot\nabla V (\alpha,x)   | \Delta^{2d} g(x)|^2  d\eta(\alpha)dx\right|\\
		&\quad\quad\quad\lesssim_{r_0}  (||\nabla g_1||_{L^\infty}^\frac{2}{8d-1}+\varepsilon) ||\Delta^{2d} g||_{\dot H^{\frac{1}{2}}}^2+\mathcal{F}(||g_2||_{ H^{4d+\frac{1}{2}}}+\varepsilon^{-1})
	\end{aligned}	
	\end{equation*}
	for any $\varepsilon\in (0,1)$.
\end{lemma}
\begin{proof} 
	Using H\"older's inequality we obtain
	\begin{equation}\label{I2decompose}
		I_2\leq\|\Delta^{2d} g\|_{L^{p_1}}^2  
		\left\|\int_{\mathbb{R}^d}\hat\alpha\cdot\nabla V (\alpha,\cdot)  d\eta(\alpha)\right\|_{L^{8d-1}}.
	\end{equation}
	Note that
	$$
	\int_{\mathbb{R}^d}\hat\alpha\cdot\nabla\left(\frac{1}{\langle\hat\alpha\cdot\nabla g(x)\rangle^{d+1}}\right)d\eta(\alpha)=0.
	$$
	Recall the definition \eqref{defV} of $V$ and Lemma \ref{Cm}, we directly have 
	\begin{equation*}
		\begin{aligned}
			&\left|\int\hat\alpha\cdot\nabla V (\alpha,x)  d\eta(\alpha) \right|\lesssim\left|\int\frac{\Delta_\alpha g(x)}{\langle \Delta_\alpha g(x)\rangle^{d+3}}\hat\alpha\cdot \nabla \Delta_\alpha g(x)  d\eta(\alpha)\right|\\&\lesssim \int\left|E_\alpha g(x)\right||\Delta_\alpha \nabla g(x)| \frac{d\alpha}{ |\alpha|^{d}}+\left|\int\frac{\hat \alpha\cdot\nabla g(x)}{\langle \hat \alpha\cdot\nabla g(x)\rangle^{d+3}}\hat \alpha \cdot\nabla \Delta_\alpha g(x)  d\eta(\alpha) \right|\\
			&=:L_1+L_2.
		\end{aligned}
	\end{equation*}
	Using H\"{o}lder's inequality one has 
	$$
	L_1
	\lesssim \left(\int|E_\alpha g(x)|^2\frac{d\alpha}{ |\alpha|^{d+1}}\right)^\frac{1}{2}\left(\int|\delta_\alpha \nabla g(x)|^2 \frac{d\alpha}{ |\alpha|^{d+1}}\right)^\frac{1}{2}.
	$$
	By standard interpolation one has 
	\begin{align}\label{Ealpha}
		\left(\int|E_\alpha g(x)|^2\frac{d\alpha}{ |\alpha|^{d+1}}\right)^\frac{1}{2}+\left(\int|\delta_\alpha \nabla g(x)|^2\frac{d\alpha}{ |\alpha|^{d+1}}\right)^\frac{1}{2}\lesssim\|\nabla g\|_{L^\infty}^\frac{1}{3}\left(\sup_{\alpha} \frac{| E_\alpha g(x)|}{|\alpha|^{\frac{3}{4}}}\right)^{\frac{2}{3}}.
	\end{align}		
	Use the condition $\|\nabla g\|_{L^\infty}\leq r_0$  we have 
	\begin{equation}\label{J1}
		\|L_1\|_{L^{8d-1}}	\lesssim_{r_0}\left\|\left(\sup_{\alpha} \frac{|\delta_\alpha \nabla g(\cdot)|}{|\alpha|^{\frac{3}{4}}}\right)^{\frac{4}{3}}\right\|_{L^{8d-1}}\lesssim_{r_0}
		\|\nabla g\|_{\dot F^{\frac{3}{4}}_{p_3,\infty}}^\frac{4}{3}.
	\end{equation}
	Similarly, we have
	$$
	L_2\lesssim \left(\sup_{\alpha}\frac{|\delta_\alpha \nabla g(x)|}{|\alpha|^{\frac{1}{2}}}\right)^{\frac{1}{3}} \left(\sup_{\alpha} \frac{| E_\alpha\nabla g(x)|}{|\alpha|^{\frac{1}{4}}}\right)^{\frac{2}{3}}.
	$$
	Applying H\"older's inequality again, one has
	\begin{equation}\label{J2}
			\begin{aligned}
		\|L_2\|_{L^{8d-1}}	&\lesssim\left\|\left(\sup_{\alpha}\frac{|\delta_\alpha \nabla g(x)|}{|\alpha|^{\frac{1}{2}}}\right)^{\frac{1}{3}}\right\|_{L^{6(8d-1)}}\left\|\left(\sup_{\alpha} \frac{| E_\alpha\nabla g(x)|}{|\alpha|^{\frac{1}{4}}}\right)^{\frac{2}{3}}\right\|_{L^{6(8d-1)/5}}\\
			&\lesssim
			\|\nabla g\|_{\dot F^{\frac{1}{2}}_{p_4,\infty}}^\frac{1}{3}	\|\nabla^2 g\|_{\dot F^{\frac{1}{4}}_{p_5,\infty}}^\frac{2}{3}.
		\end{aligned}
	\end{equation}

	By the Gagliardo-Nirenberg interpolation inequality \eqref{GNinterpolation}, we have 
	$$
	\|\nabla^2 g\|_{\dot F^{\frac{1}{4}}_{p_5,\infty}}\lesssim\|\nabla g\|_{L^\infty}^\frac{16d-7}{16d-2}\|\Delta^{2d}g\|_{\dot H^\frac{1}{2}}^\frac{5}{16d-2}.
	$$
	Combining this with \eqref{GNineqs}, \eqref{I2decompose}, \eqref{J1} and \eqref{J2}, we obtain
	\begin{align*}
		I_2&\lesssim_{r_0}\left(\|\nabla g_1\|_{L^\infty}^\frac{2}{8d-1}\|\Delta^{2d}g_1\|_{\dot H^\frac{1}{2}}^\frac{16d-4}{8d-1}+\|\nabla g_2\|_{L^\infty}^\frac{1}{8d-1}\|\Delta^{2d}g_2\|_{\dot H^\frac{1}{2}}^\frac{8d-2}{8d-1}\right)\|\Delta^{2d}g\|_{\dot H^\frac{1}{2}}^{\frac{2}{8d-1}}\\
		&\lesssim_{r_0}\|\nabla g_1\|_{L^\infty}^\frac{2}{8d-1}(\|\Delta^{2d}g\|_{\dot H^\frac{1}{2}}^2+\|\Delta^{2d}g_2\|_{\dot H^\frac{1}{2}}^\frac{16d-4}{8d-1}\|\Delta^{2d}g\|_{\dot H^\frac{1}{2}}^{\frac{2}{8d-1}})\\
		&\quad\quad\quad\quad+\|\Delta^{2d}g_2\|_{\dot H^\frac{1}{2}}^\frac{8d-2}{8d-1}\|\Delta^{2d}g\|_{\dot H^\frac{1}{2}}^{\frac{2}{8d-1}}.
	\end{align*}			
	Applying Young's inequality, we get
	$$
	I_2\lesssim_{r_0}(\|\nabla g_1\|_{L^\infty}^\frac{2}{8d-1}+\varepsilon)\|\Delta^{2d}g\|_{\dot H^\frac{1}{2}}^2++\mathcal{F}(||g_2||_{ H^{4d+\frac{1}{2}}}+\varepsilon^{-1}).
	$$
\end{proof}

\begin{lemma}\label{le4}(Estimate for $I_3$)\\
	Let $g$, $r_0$ as defined in Proposition \ref{thm1}, and $M$ as defined in \eqref{defM}, there holds
	\begin{equation*}
		\begin{aligned}
				&\left|\iint  M(\alpha,x) d\eta(\alpha) \Delta^{2d} g(x) dx\right|\\
				&\quad\quad\quad\quad\lesssim_{r_0}  (||\nabla g_1||_{L^\infty}^{\frac{1}{8d-1}}+\varepsilon) ||\Delta^{2d} g||_{\dot H^{\frac{1}{2}}}^2+\mathcal{F}(||g_2||_{ H^{4d+\frac{1}{2}}}+\varepsilon^{-1})
		\end{aligned}
	\end{equation*}
	for any $\varepsilon\in (0,1)$.
\end{lemma}
\begin{proof}
	Applying H\"{o}lder's inequality one has		
	\begin{align*}
		\left|\iint  M(\alpha,x) d\eta(\alpha) \Delta^{2d} g(x) dx\right|\lesssim ||\Delta^{2d} g||_{L^{p_1}} \left\|\int M(\alpha,x) d\eta(\alpha) \right\|_{L ^{\frac{8d-1}{4d}}}.
	\end{align*}
	Recall the definition of $M(\alpha,x)$ in \eqref{defM},
	we have
	\begin{align*}
		|M(\alpha,x)|&\lesssim \sum_{m_1+m_2=4d,m_2>0}|\alpha||\Delta_\alpha D^{1+m_1}g(x)|\left|D^{m_2}\left(\frac{1}{\left\langle \Delta_\alpha g\right\rangle^{d+1}}\right)\right|\\&\lesssim
		\sum_{m_1+m_2=4d,m_2>0}\sum_{k=1}^{m_2}|\alpha||\Delta_\alpha D^{1+m_1}g(x)||D^k\Delta_\alpha g(x)|^{\frac{m_2}{k}}.
	\end{align*}
	Applying H\"{o}lder's inequality one obtains
	\begin{align*}
		&\left\|\int M(\alpha,x) d\eta(\alpha) \right\|_{L^{\frac{8d-1}{4d}}}\\
		&\quad\quad\lesssim
		\sum_{m_1+m_2=4d,m_2>0}\sum_{k=1}^{m_2}\left\|\int |\Delta_\alpha D^{1+m_1}g||D^k\Delta_\alpha g|^{\frac{m_2}{k}}\frac{d\alpha}{|\alpha|^{d-1}}\right\|_{L^{\frac{8d-1}{4d}}}\\
		&\quad\quad\lesssim\sum_{m_1+m_2=4d,m_2>0}\sum_{k=1}^{m_2}\|D^{1+m_1}g\|_{\dot F^{\frac{1}{2}}_{p_6,2}}\|D^jg\|_{\dot F^{\beta}_{p_7,2m_2/k}}^\frac{m_2}{j},
	\end{align*}
	where
	\[
	{p_6}={\frac{2(8d-1)}{2m_1+1}},\qquad{p_7}={\frac{2(8d-1)m_2}{k(2m_2-1)}}, \qquad \beta=1-\frac{k}{2m_2}.
	\]
	By the interpolation inequality  \eqref{GNinterpolation}, we know that
	\begin{align*}
		\|D^{1+m_1}g\|_{\dot F^{\frac{1}{2}}_{p_6,2}}&\lesssim\|\nabla g\|_{L^\infty}^{\frac{8d-1-2m_1-1}{8d-1}}\|\Delta^{2d} g\|_{\dot H^{\frac{1}{2}}}^{\frac{2m_1+1}{8d-1}},\\	
		\|D^kg\|_{\dot F^{\beta}_{p_7,2m_2/k}}&\lesssim\|\nabla g\|_{L^\infty}^{1-\frac{(2m_2-1)k}{(8d-1)m_2}}\|\Delta^{2d} g\|_{\dot H^{\frac{1}{2}}}^{\frac{(2m_2-1)k}{(8d-1)m_2}}.
	\end{align*}
	Then one has
	\begin{align*}
		\left\|\int M(\alpha,x) d\eta(\alpha) \right\|_{L^{\frac{8d-1}{4d}}}\lesssim\|\Delta^{2d} g\|_{\dot H^{\frac{1}{2}}}^{\frac{8d}{8d-1}}\|\nabla g\|_{L^\infty}^{\frac{8d-2}{8d-1}}.
	\end{align*}
	Combining this with \eqref{GNineqs} we have
	\begin{align*}
		I_3&\lesssim \left(\|\nabla g_1\|_{L^\infty}^\frac{1}{8d-1}\|\Delta^{2d}g_1\|_{\dot H^\frac{1}{2}}^\frac{8d-2}{8d-1}+\|\nabla g_2\|_{L^\infty}^\frac{1}{8d-1}\|\Delta^{2d}g_2\|_{\dot H^\frac{1}{2}}^\frac{8d-2}{8d-1}\right)
		\|\Delta^{2d} g\|_{\dot H^{\frac{1}{2}}}^{\frac{8d}{8d-1}}\|\nabla g\|_{L^\infty}^{\frac{8d-2}{8d-1}}\\
		&\lesssim_{r_0}\|\nabla g_1\|_{L^\infty}^\frac{1}{8d-1}(\|\Delta^{2d}g\|_{\dot H^\frac{1}{2}}^2+\|\Delta^{2d}g_2\|_{\dot H^\frac{1}{2}}^\frac{8d-2}{8d-1}\|\Delta^{2d}g\|_{\dot H^\frac{1}{2}}^{\frac{8d}{8d-1}})+\|\Delta^{2d}g_2\|_{\dot H^\frac{1}{2}}^\frac{8d-2}{8d-1}\|\Delta^{2d}g\|_{\dot H^\frac{1}{2}}^{\frac{8d}{8d-1}}.
	\end{align*}
	Applying Young's inequality we have
	\begin{align*}
		I_3\lesssim_{r_0}(||\nabla g_1||_{L^\infty}^\frac{1}{8d-1}+\varepsilon) ||\Delta^{2d} g||_{\dot H^{\frac{1}{2}}}^2+\mathcal{F}(||g_2||_{ H^{4d+\frac{1}{2}}}+\varepsilon^{-1}),
	\end{align*}
	which completes the proof.
\end{proof}
\vspace{0.5cm}
\begin{lemma}\label{le3}(Estimate for $I_4$)\\
	Let $g$ as defined in Proposition \ref{thm1}, and $V$ as defined in \eqref{defV}, there holds
	\begin{align*}
		&\left|\iint\Delta_\alpha (\Delta^{2d} g)(x)V(\alpha,x)d\eta(\alpha)\Delta^{2d} g(x) dx\right|\\
		&\quad\quad\lesssim_{r_0}  (||\nabla g_1||_{L^\infty}^\frac{1}{8d-1}+\varepsilon) ||\Delta^{2d} g||_{\dot H^{\frac{1}{2}}}^2+\mathcal{F}(||g_2||_{ H^{4d+\frac{1}{2}}}+\varepsilon^{-1})
	\end{align*}
	for any $\varepsilon\in (0,1)$.
\end{lemma}
\begin{proof} 
	Applying Lemma \ref{Cm}, H\"{o}lder's inequality and \eqref{Ealpha} one has
	\begin{align*}
		I_4&\lesssim ||\Delta^{2d} g||_{L^{p_1}} \|\Delta^{2d} g\|_{\dot H^\frac{1}{2}}\|\nabla g\|_{\dot F_{p_3,\infty}^\frac{3}{4}}^\frac{2}{3}\|\nabla g\|_{L^\infty}^\frac{1}{3}.
	\end{align*}	
	By the interpolation inequalities \eqref{GNineqs}	we obtain
	\begin{align*}
		I_4&\lesssim_{r_0}\left(\|\nabla g_1\|_{L^\infty}^\frac{1}{8d-1}\|\Delta^{2d}g_1\|_{\dot H^\frac{1}{2}}^\frac{8d-2}{8d-1}+\|\nabla g_2\|_{L^\infty}^\frac{1}{8d-1}\|\Delta^{2d}g_2\|_{\dot H^\frac{1}{2}}^\frac{8d-2}{8d-1}\right)\|\Delta^{2d}g\|_{\dot H^\frac{1}{2}}^{\frac{8d}{8d-1}}\\
		&\lesssim_{r_0}\|\nabla g_1\|_{L^\infty}^\frac{1}{8d-1}(\|\Delta^{2d}g\|_{\dot H^{\frac{1}{2}}}^2+\|\Delta^{2d}g_2\|_{\dot H^{\frac{1}{2}}}^\frac{8d-2}{8d-1}\|\Delta^{2d}g\|_{\dot H^{\frac{1}{2}}}^{\frac{8d}{8d-1}})\\
		&\quad\quad\quad\quad+\|\Delta^{2d}g_2\|_{\dot H^{\frac{1}{2}}}^\frac{8d-2}{8d-1}\|\Delta^{2d}g\|_{\dot H^{\frac{1}{2}}}^{\frac{8d}{8d-1}}.
	\end{align*}
	Applying Young's inequality we have
	\begin{align*}
		I_4\lesssim_{r_0}(||\nabla g_1||_{L^\infty}^\frac{1}{8d-1}+\varepsilon) ||\Delta^{2d} g||_{\dot H^{\frac{1}{2}}}^2+\mathcal{F}(\|\Delta^{2d} g_2\|_{ H^{\frac{1}{2}}}+\varepsilon^{-1}),
	\end{align*}
	which completes the proof.
\end{proof}
\bibliographystyle{amsplain}

\end{document}